\documentclass[11pt,a4paper]{article}

\usepackage[utf8]{inputenc}
\usepackage[T1]{fontenc}
\usepackage[english]{babel}
\usepackage{amsmath,amssymb,amsthm,mathtools}
\usepackage{geometry}
\usepackage{hyperref}
\usepackage{enumitem}
\usepackage{xcolor}
\newcommand{\R}{\mathbb{R}}
\newcommand{\LL}{\mathcal{L}}
\newcommand{\HH}{\mathcal{H}}
\newcommand{\diver}{\operatorname{div}}
\newcommand{\Lip}{\operatorname{Lip}}
\newcommand{\supp}{\operatorname{spt}}
\newcommand{\Gr}{\operatorname{Gr}}
\newcommand{\osc}{\operatorname{osc}}
\newcommand{\Id}{\operatorname{Id}}
\newcommand{\diam}{\operatorname{diam}}

\newtheorem{thm}{Theorem}[section]
\newtheorem{prop}[thm]{Proposition}
\newtheorem{lem}[thm]{Lemma}
\newtheorem{cor}[thm]{Corollary}
\theoremstyle{definition}
\newtheorem{defn}[thm]{Definition}
\theoremstyle{remark}
\newtheorem{rem}[thm]{Remark}

\title{Lipschitz solvability of prescribed Jacobian and divergence for singular measures}
\author{Luigi De Masi \& Andrea Marchese}
\date{}

\begin{document}
\maketitle

\begin{abstract}
Let $\mu$ be a finite Radon measure on an open set $\Omega\subset\R^d$, singular with respect to the Lebesgue measure.
We prove Lusin-type solvability results for the prescribed divergence equation and the prescribed Jacobian equation with Lipschitz solutions.
More precisely, for every $\varepsilon>0$ and every Borel datum $f \colon \Omega \to \mathbb{R}$ there exists a vector field $V\in C^1_c(\Omega;\R^d)$ such that
$\diver V=f$
on a compact set $K\subset\Omega$ with $\mu(\Omega\setminus K)<\varepsilon$, and {$\Lip(V)\le (1+\varepsilon)\|f\|_{L^\infty(\Omega,\mu)}$}.
Similarly, for every Borel datum $g\colon \Omega \to \mathbb{R}$ there exists a map $\Phi$ with $\Phi-\Id\in C^1_c(\Omega;\R^d)$ such that
$\det D\Phi=g$
on a compact set $K\subset\Omega$ with $\mu(\Omega\setminus K)<\varepsilon$, and $\Lip(\Phi-\Id)\le (1+\varepsilon)\|g-1\|_{L^\infty(\Omega,\mu)}$.
The maps $V$ and $\Phi-\Id$ can be chosen arbitrarily small in supremum norm.
% For sufficiently small Jacobian data we obtain global bi-Lipschitz $C^1$ diffeomorphisms of $\R^d$.
\end{abstract}



\section{Introduction}

In \cite[Corollary 1.5]{ABM}, Alberti, Bate and the second named author proved that when a Radon measure $\mu$ is singular with respect to the Lebesgue measure, several classical differential operators are not closable from the space of Lipschitz functions to $L^p(\mu)$. The list includes gradient, divergence, and the Jacobian determinant.
In particular, there exist equi-Lipschitz functions $u_n$ converging uniformly to zero whose derivatives remain large in some fixed direction on a set of positive $\mu$-measure and equi-Lipschitz vector fields $V_n$ converging uniformly to zero whose divergences remain large on a set of positive $\mu$-measure.
For the Jacobian determinant, one instead obtains equi-Lipschitz perturbations of the identity $\Phi_n$ for which $\det D\Phi_n-1$ remains large on a set of positive $\mu$-measure.

The purpose of the present paper is to show that the same geometric mechanism can be exploited in a constructive way.
We prove Lusin-type solvability results for the equations\footnote{The corresponding statement for the gradient was proved in \cite{MS}; in \cite{DMM} it was strengthened to yield solutions with uniform Lipschitz bounds, at the cost of additional (necessary) assumptions on the datum.}
$$
\diver V=f
\qquad\text{and}\qquad
\det D\Phi=g
$$
with Lipschitz solutions, where the data are only bounded with respect to a singular measure $\mu$. 
The solutions are exact on sets of arbitrarily large $\mu$-measure and enjoy uniform Lipschitz bounds depending only on the supremum norm of the datum.
More precisely, our main results are the following. The symbol $\mu\perp\LL^d$ means that $\mu$ is singular with respect to the Lebesgue measure $\LL^d$, that is, there exists a Borel set $E\subset\R^d$ such that $\LL^d(E)=0=\mu(\R^d\setminus E)$. 

\begin{thm}[Lusin solvability for divergence]\label{thm:div}
Let $\mu$ be a finite Radon measure on an open set $\Omega\subset\R^d$ with $\mu\perp\LL^d$, let $f:\Omega\to\R$ be a Borel function and $\delta>0$.
Then, for every $\varepsilon>0$, there exist a compact set $K\subset\Omega$ and a vector field $V\in C^1_c(\Omega;\R^d)$ such that
$$
\mu(\Omega\setminus K)<\varepsilon,
\qquad
\diver V(x)=f(x)\quad\forall x\in K,
$$
and
$$
\Lip(V)\le (1+\delta) \|f\|_{L^\infty(\Omega,\mu)},
\qquad
\|V\|_{C^0(\Omega)}\le \varepsilon.
$$
\end{thm}

\begin{thm}[Lusin solvability for the Jacobian]\label{thm:jac}
Let $\mu$ be a finite Radon measure on an open set $\Omega\subset\R^d$ with $\mu\perp\LL^d$, let $g:\Omega\to\R$ be a Borel function and $\delta>0$.
Then, for every $\varepsilon>0$, there exist a compact set $K\subset\Omega$ and a map $\Phi$ with $\Phi-\Id\in C^1_c(\Omega;\R^d)$ such that
$$
\mu(\Omega\setminus K)<\varepsilon,
\qquad
\det D\Phi(x)=g(x)\quad\forall x\in K,
$$
and
$$
\Lip(\Phi-\Id)\le (1+\delta)\|g-1\|_{L^\infty(\Omega,\mu)},
\qquad
\|\Phi-\Id\|_{C^0(\Omega)}\le \varepsilon.
$$
In particular, if $(1+\delta)\|g-1\|_{L^\infty(\Omega,\mu)}< 1$, then $\Phi$ is a global diffeomorphism.
\end{thm}

A key feature is that the solutions in Theorem \ref{thm:div} can be chosen arbitrarily small in supremum norm.
For the divergence equation this also yields perturbations of arbitrary reference vector fields, see Remark \ref{rem:background-div}.
In the Jacobian case our construction is naturally formulated as an arbitrarily small perturbation of the identity map whose Jacobian is prescribed on sets of almost full $\mu$-measure, although it is possible to perturb a given diffeomorphism as well, as stated below.
% see Remark \ref{rem:jac-id}.

\begin{cor}\label{cor:diff_jac}
Let $\mu$, $\Omega$, and $g$ be as in Theorem \ref{thm:jac}, and let
$F \colon \Omega \to \Sigma \subseteq \mathbb{R}^d$ be a diffeomorphism.
Then, for every $\delta,\varepsilon>0$, there exist a compact set
$K\subset\Omega$ and a map $\Phi\in C^1(\Omega;\Sigma)$ with
$\Phi-F\in C^1_c(\Omega;\mathbb{R}^d)$ such that
$$
\mu(\Omega\setminus K)<\varepsilon,
\qquad
\det D\Phi(x)=g(x)\quad\forall x\in K,
$$
and
$$
\Lip(\Phi-F)\le (1+\delta)\Lip(F)
\left\|\frac{g}{\det DF}-1\right\|_{L^\infty(\Omega,\mu)},
\qquad
\|\Phi-F\|_{C^0(\Omega)}\le \varepsilon.
$$
\end{cor}

The main ingredient in the proofs is the construction of width functions, invented by Alberti, Cs\"ornyei and Preiss and already used as a main building block in the refined Lusin-type theorem for gradients proved in \cite{DMM}.
The common core of our proof is a directional scalar lemma: under a suitable cone-invisibility condition, a bounded continuous scalar function can be realized as a directional derivative outside a set of arbitrarily small $\mu$-measure, with a uniform Lipschitz bound and with arbitrarily small supremum norm.
Once this is established, the divergence and Jacobian equations follow from standard measure theory arguments and simple algebraic computations.

\begin{rem}[Rectifiable measures as a model case]
The geometric mechanism behind our construction is particularly transparent when the measure is concentrated on a rectifiable set.
Indeed, suppose that $\mu$ is carried by a $k$-rectifiable set with $k<d$.
Locally, such a set admits approximate tangent planes, and one may choose directions transverse to those planes.
On each rectifiable patch, the divergence and Jacobian identities used in this paper suggest that one should solve a scalar problem by prescribing a directional derivative in a transverse direction.

This should be contrasted with the situation for Lusin-type $C^1$ approximation:
by \cite{MMrect}, the Lusin-type approximation of Lipschitz functions by $C^1$ functions characterizes a much more rigid class of measures, namely those that can be decomposed into rectifiable pieces, possibly of different dimensions.
Thus, from the point of view of approximation by $C^1$ functions, measures of rectifiable type behave genuinely differently from generic singular measures.

For the prescription of divergence or Jacobian, however, rectifiability plays a different role.
Even in the rectifiable setting, for a general bounded or continuous datum one still needs the directional scalar lemma in order to keep the solution arbitrarily small in supremum norm while controlling its Lipschitz constant.
Rectifiability mainly simplifies the geometry, since the transverse directions are then immediately available, but it does not change the analytic nature of the problem.
In particular, the possibility of prescribing divergence or Jacobian is not specific to rectifiable measures: the same mechanism remains available for arbitrary singular measures.

The main point of the present work is that one can recover, for a general singular measure, the same type of transverse construction that is geometrically evident in the rectifiable case, even in the absence of any global rectifiable support or globally defined normal directions.
The decomposability bundle plays the role of a measurable tangent structure, and the cone-null/width-function argument provides an effective way to exploit directions transverse to it.
\end{rem}

\begin{rem}[Endpoint estimates and the flat chain conjecture]

The Lusin-type solvability of the prescribed divergence and prescribed Jacobian equation, see Theorem~\ref{thm:div} and Theorem~\ref{thm:jac}, should be compared with the classical failure of endpoint bounded-to-Lipschitz estimates for first-order equations, in the spirit of Ornstein's non-inequality \cite{Ornstein}.
The failure of such estimates for the prescribed divergence equation
$\operatorname{div} u = f,
$ was shown by Preiss \cite{Preiss1997} and, independently and by different methods, by McMullen \cite{McMullen1998}, see also the discussion in \cite{DFT03} and the recent work \cite{DePauw2026}. 
Tak\'a\v{c}'s counterexample \cite{Takac} shows that this obstruction for the prescribed Jacobian equation, $\det D\phi = f$ invalidates Lang's formulation of the flat chain conjecture, which requires pointwise control.

Recall that Ambrosio and Kirchheim \cite{AK00} developed a theory of metric currents in complete metric spaces, extending the classical Federer-Fleming theory beyond the Euclidean setting.
A basic open question in their theory is the so called \emph{flat chain conjecture}; namely whether, in Euclidean spaces, metric currents with compact support coincide with Federer-Fleming flat chains with finite mass; see \cite[Section~11]{AK00}.
The conjecture is known to be true in dimension $k=1$, and in the top-dimensional case $k=d$ it was proved by De Philippis and Rindler \cite[Theorem~1.15]{DPR}.

The second named author and Merlo recently pointed out that a possible route towards the intermediate-dimensional cases $2\le k\le d-1$ would be a suitable Lusin-type solvability theory for prescribed differential forms; more precisely, the argument shows that the conjecture for $k$-currents would follow from an $L^\infty$-to-Lipschitz solvability result for the equation
$d\varphi=\omega$
on sets of arbitrarily large measure, for $k$-forms $\omega$ satisfying the natural orthogonality condition with respect to the bundle $V^k(\mu,\cdot)$; see \cite[Remark~4.1]{MM}.
In \cite[Conjecture~4.1 and Section~4]{DMM} this idea is formulated explicitly, and it is proved there that a positive solution of that conjecture would imply the Ambrosio-Kirchheim flat chain conjecture in full generality.

By contrast with Lang's formulation, the Ambrosio-Kirchheim point of view is naturally compatible with Lusin-type statements, where the relevant differential constraint is imposed only on a set carrying almost all the mass of a singular measure.
In top dimension, the differential-form equation
$d\varphi=\omega$ reduces in coordinates to a prescribed divergence equation: if one identifies a $(d-1)$-form with a vector field $V$, then $d\varphi$ is exactly $(\diver V)\,dx_1\wedge\cdots\wedge dx_d$.
For this reason, the top-dimensional case of the Lusin-type strategy proposed in \cite[Remark~4.1]{MM} and formulated in \cite[Conjecture~4.1 and Section~4]{DMM} is precisely the solvability of the prescribed divergence equation on sets carrying arbitrarily large $\mu$-mass.
Our theorem of Lusin solvability for divergence provides exactly this type of test in the singular setting, and this clarifies why that is the natural substitute for pointwise endpoint estimates in the top-dimensional flat chain conjecture and identifies precisely the mechanism by which, despite the failure of such pointwise estimates, the flat chain conjecture follows in top dimension.
\end{rem}

\begin{rem}[Comparison with Dacorogna-Moser]

The classical Dacorogna-Moser theorem provides smooth solutions to the transport problem between two absolutely continuous measures.
More precisely, if $f,g$ are positive sufficiently regular densities on a bounded smooth domain $\Omega\subset\R^d$ with
$\int_\Omega f=\int_\Omega g,$
then there exists a diffeomorphism $\Phi:\Omega\to\Omega$ such that
\begin{equation}\label{eq:transp}
g(\Phi(x))\,\det D\Phi(x)=f(x)
\qquad\text{for all }x\in\Omega,    
\end{equation}
equivalently,
$\Phi_\#(f\,dx)=g\,dx;$
see \cite{DM90}.
Our result suggests an analogous picture for singular measures, but with an important distinction.

For absolutely continuous measures, the Jacobian determinant describes the infinitesimal distortion of the ambient volume under a change of variables, and the transport equation between two densities takes the form \eqref{eq:transp}.

By contrast, Theorem~\ref{thm:jac} is still an ambient Jacobian statement, and by itself it does not imply a direct push-forward relation of the form
$\Phi_\#\mu = g\mu$
for an arbitrary singular measure $\mu$.
What the theorem does provide is a very regular ambient deformation of the space whose Jacobian is prescribed on sets of arbitrarily large $\mu$-measure.

In this sense, the theorem is about constructing a global change of coordinates whose Jacobian has the prescribed value on a set of arbitrarily large $\mu$-measure, rather than about transporting the measure $\mu$ itself.
\end{rem}
\section*{Acknowledgments}

This work was supported by the Italian Ministry of University and Research (MUR) through the FIS 2 project \textit{SingMeas:"Singular Structures in the Geometry of Measures: decompositions, rigidity, and rectifiability"}, project code FIS-2023-0272ff5 (CUP: E53C25001800001).

\section{Preliminaries}

Throughout the paper, sets and functions are Borel measurable.
Measures are finite, positive and Radon.
If $\mu$ is a measure and $E$ is Borel, $\mu\llcorner E$ denotes the restriction of $\mu$ to $E$.
For $x\in\R^d$ and $r>0$, $B(x,r)$ denotes the open Euclidean ball of center $x$ and radius $r$.

\subsection{The decomposability bundle}

We recall only the part of the theory that is needed in what follows.
Let $\mu$ be a positive Radon measure on $\R^d$.
The decomposability bundle is a Borel map
$$
V(\mu,\cdot):\R^d\to \bigcup_{k\le d}\Gr(k,d),
$$
where $\Gr(k,d)$ is the Grassmannian of unoriented linear subspace of $\mathbb{R}^d$ of dimension $k$. We refer to \cite{AM16} for the definition and for the structural properties.

The crucial fact concerning this tool which is used here is that singular measures do not have full decomposability bundle almost everywhere.

\begin{prop}[{\cite[Theorem 2.4]{ABM}}]\label{prop:nonfullbundle}
Let $\mu$ be a finite Radon measure on $\R^d$.
Then
$$
V(\mu,x)=\R^d \quad\text{for $\mu$-a.e.\ }x
\qquad\Longleftrightarrow\qquad
\mu\ll\LL^d.
$$
In particular, if $\mu\perp\LL^d$, then $V(\mu,x)\neq \R^d$ for $\mu$-a.e.\ $x \in \mathbb{R}^d$.
\end{prop}

\subsection{Cones and cone-null sets}

For $e\in \mathbb{S}^{d-1}$ and $\alpha\in(0,\pi/2)$ we set
$$
C(e,\alpha)\coloneqq \{v\in\R^d:\ v\cdot e\ge \cos\alpha\,|v|\}
$$
the closed cone with axis $e$ and opening angle $\alpha$.
% We also write $Z_e\coloneqq e^\perp$ for the orthogonal complement of the vector subspace generated by $e$.

The following statement allows us to replace a pointwise choice of a direction transverse to the decomposability bundle by a finite family of constant directions. We include the elementary geometric proof for the reader convenience.

\begin{lem}\label{lem:finite-directions}
For every $\alpha \in \left(0,\frac{\pi}{2}\right)$, there exist unit vectors $v_1,\dots,v_N\in \mathbb{S}^{d-1}$, depending only on $\alpha$ and the dimension $d$, with the following property:
for every proper linear subspace $L\subsetneq \R^d$ there exists at least one index $j$ such that
$$
L\cap C\left(v_j, \alpha\right)=\{0\}.
$$
\end{lem}

\begin{proof}
Choose a finite geodesic $\frac{1}{2}\left(\frac{\pi}{2}- \alpha\right)$-net $\{v_1,\dots,v_N\}\subset \mathbb{S}^{d-1}$, meaning that for every $n\in \mathbb{S}^{d-1}$ there exists $j$ such that
$$
n \cdot v_j \geq \cos \left( \frac{1}{2} \left(\frac{\pi}{2}- \alpha \right)\right).
$$
Let $L\subsetneq\R^d$ be a proper subspace and choose a unit vector $n\in L^\perp$. Then, by the definition of the net, there exists an index $j$ such that $n\cdot v_j\ge \cos \left( \frac{1}{2} \left(\frac{\pi}{2}- \alpha \right)\right)$. Thus, for every $v \in L \setminus \{0\}$,  it holds
$$
(v \cdot v_j)^2
=
\left( v \cdot \big(v_j - (v_j \cdot n) n\big) \right)^2
\leq
|v|^2\big|v_j - (v_j \cdot n) n\big|^2
=
|v|^2 \big( 1- (v_j \cdot n)^2\big)
\leq
|v|^2 \sin^2\left( \frac{1}{2} \left(\frac{\pi}{2}- \alpha \right)\right),
$$
which implies $v \notin C\left(v_j, \alpha\right)$, since $\sin\left( \frac{1}{2} \left(\frac{\pi}{2}- \alpha \right)\right) = \sqrt\frac{1-\sin \alpha}{2}< \cos \alpha$ for $\alpha \in \left(0,\frac{\pi}{2}\right)$.
\end{proof}

\begin{defn}
A Borel set $E\subset\R^d$ is called \emph{$C(e,\alpha)$-null} if for every Lipschitz curve $\gamma:[0,1]\to\R^d$ such that $\dot\gamma(t)\in C(e,\alpha)$ for a.e.\ $t \in [0,1]$, one has
$$
\HH^1\bigl(E\cap \gamma([0,1])\bigr)=0.
$$
\end{defn}

\begin{rem}\label{rem:minuscone}
If $L\subset\R^d$ is a linear subspace and $L\cap C(e,\alpha)=\{0\}$, then also $L\cap C(-e,\alpha)=\{0\}$.
Indeed, if $v\in L\cap C(-e,\alpha)$, then $-v\in L\cap C(e,\alpha)$.
\end{rem}

The next lemma is one of the main local tools. See \cite[Lemma 7.5]{AM16} for a proof and also \cite[Lemma 2.3]{DMM}.

\begin{lem}[Concentration on cone-null sets]\label{lem:Cnull}
Let $B\subset\R^d$ be a Borel set, let $\mu$ be a Radon measure, and let $e\in \mathbb{S}^{d-1}$, $\alpha\in(0,\pi/2)$ be such that
$$
V(\mu,x)\cap C(e,\alpha)=\{0\}
\qquad\text{for $\mu$-a.e.\ }x\in B.
$$
Then there exists a $C(e,\alpha)$-null Borel set $F\subset B$ such that $\mu(B\setminus F)=0$.
\end{lem}

\subsection{Width functions}

The second local tool is the existence of a so-called width function. See \cite[Lemma 4.12]{AM16} for a proof and also \cite[Lemma 2.4]{DMM}.

\begin{lem}[Width function]\label{lem:width}
Let $e\in \mathbb{S}^{d-1}$, $\alpha\in(0,\pi/2)$ and let $E\subset\R^d$ be compact and $C(e,\alpha)$-null.
Then for every $\zeta>0$ there exists a function $\varphi\in C^\infty(\R^d)$ such that:
\begin{enumerate}[label=\rm(\roman*),leftmargin=2em]
\item $0\le \varphi(x)\le \zeta$ for every $x \in \R^d$;
\item $0\le \partial_e\varphi(x)\le 1$ for every $x \in \R^d$ and $\partial_e\varphi(x)=1$ for all $x \in E$;
\item $|\partial_v\varphi(x)| \le \dfrac{1}{\tan\alpha}$ for every $v \in e^\perp$ with $|v|=1$ and every $x \in \R^d$.
% \item $|d_{Z_e}\varphi(x)| \le \dfrac{1}{\tan\alpha}$ for every $x \in \R^d$.
\end{enumerate}
\end{lem}

\begin{rem}[Gradient bound]\label{rem:gradphi}
A direct consequence of Lemma~\ref{lem:width} is that
$$
|D\varphi(x)|
% \le |\partial_e\varphi(x)| + |d_{Z_e}\varphi(x)|
\le 1+\frac{1}{\tan\alpha}
\qquad\forall x\in\R^d.
$$
We set
$$
c_\alpha\coloneqq 1+\frac{1}{\tan\alpha}.
$$
\end{rem}

\section{A local scalar property}

The common core of our construction is the following local scalar result.

\begin{prop}\label{prop:local}
Let $U\subset\R^d$ be open, let $\nu$ be a finite Radon measure on $U$, let $e\in \mathbb{S}^{d-1}$, let $\alpha\in(0,\pi/2)$,
and let $\lambda\in C(U)\cap L^\infty(U, \nu)$.
Assume that there exists a $C(e,\alpha)$-null set $F\subset U$ such that $\nu(U\setminus F)=0$.
Then for every $\eta>0$ and every $\delta>0$ there exist a compact set $K\subset U$ and a function $u\in C^1_c(U)$ such that
$$
\nu(U\setminus K)<\eta,
\qquad
\partial_e u(x)=\lambda(x)\quad\forall x\in K,
$$
and
$$
\|Du\|_{C^0(U)}\le (1+\delta)c_\alpha \|\lambda\|_{L^\infty(U,\nu)},
\qquad
\|u\|_{C^0(U)}\le \eta.
$$
\end{prop}

\begin{proof}
 Set $M \coloneqq \|\lambda\|_{L^\infty(U,\nu)}$.
If $M=0$, taking $u\equiv 0$ and any compact set $K\subset U$ with $\nu(U\setminus K)<\eta$ gives the conclusion.
Hence, we may assume $M>0$.
Fix $t\in \left(0,\frac{1}{2}\right)$ such that
\begin{equation}\label{eq:t_geom_piccola}
    \sum_{n=0}^{+\infty}t^n + \sum_{n=0}^{+\infty}t^{2n+4} =\frac{1}{1-t} + \frac{t^4}{1-t^2}< 1+\min\{\delta,\eta\}.
\end{equation}

By inner regularity, choose a compact set $K_0\subset F$ such that
$$
\nu(U\setminus K_0)<\frac{\eta}{2}.
$$
Since $K_0\Subset U$, there exists an open set $W\subset U$ with
$$
K_0\Subset W\Subset U.
$$
We will construct inductively a decreasing sequence of compact subsets of $W$
$$
K_0\supset K_1\supset K_2\supset\cdots
$$
and, for every $n\geq 0$, a sequence of functions $u_{n+1}\in C^1_c(W)$
such that, setting
$$
% U_0\coloneqq 0, \qquad U_n\coloneqq \sum_{m=1}^{n}u_{m}\qquad
r_n\coloneqq \lambda-\partial_e \sum_{k=1}^n u_k,
$$
the following estimates hold for every $n\ge 0$:
\begin{align}
\nu(K_n\setminus K_{n+1}) &< \frac{\eta}{2^{n+2}}, \label{eq:loss}\\
\|r_n\|_{C^0(K_n)} &\le t^n M, \label{eq:residual}\\
\|Du_{n+1}\|_{C^0(U)} &\le t^n M\bigl(c_\alpha+t^{n+4}\bigr), \label{eq:gradstep}\\
\|u_{n+1}\|_{C^0(U)} &\le \tau\,t^{2n+4}M, \label{eq:supstep}
\end{align}
where
$$
\tau\coloneqq \min\left\{1,\frac{\delta}{M}, \frac{\eta}{M}\right\}.
$$
Since $r_0=\lambda$, then \eqref{eq:residual} is true at $n=0$.\\

Assume by induction that $K_n$ and $u_n$ are already defined and satisfy the estimates.
Since $r_n$ is continuous on the compact set $K_n$, it is uniformly continuous, namely there exists $\rho_n>0$ such that
$$
|r_n(x)-r_n(y)|\le t^{n+1}M
\qquad\forall x,y\in K_n,\ |x-y|<\rho_n.
$$
Choose a finite covering
$$
K_n\subset \bigcup_{i=1}^{N_n} B(x_{n,i},\rho_n/2).
$$
Define a Borel partition of $K_n$ by
$$
E_{n,1}\coloneqq K_n\cap B(x_{n,1},\rho_n/2),
$$
and for $i\ge 2$,
$$
E_{n,i}\coloneqq \Bigl(K_n\cap B(x_{n,i},\rho_n/2)\Bigr)\setminus \bigcup_{m<i}E_{n,m}.
$$
Then the sets $E_{n,i}$ are pairwise disjoint, their union is $K_n$ and $\diam(E_{n,i})\leq \rho_n$ for every $i$.
Hence the oscillation of $r_n$ on $E_{n,i}$ satisfies
$$
\osc_{E_{n,i}}(r_n)\le t^{n+1}M.
$$
By inner regularity, for every $i$ choose a compact set $L_{n,i}\subset E_{n,i}$ such that $\nu(E_{n,i}\setminus L_{n,i})<\frac{\eta}{2^{n+2}N_n}$ and set
$$
K_{n+1}\coloneqq \bigcup_{i=1}^{N_n}L_{n,i}.
$$
Then $K_{n+1}$ is compact, $K_{n+1}\subset K_n$, and
$$
\nu(K_n\setminus K_{n+1})
=\sum_{i=1}^{N_n}\nu(E_{n,i}\setminus L_{n,i})
< \frac{\eta}{2^{n+2}}.
$$
This proves \eqref{eq:loss}.

Since the compact sets $L_{n,i}$ are pairwise disjoint and contained in $W$, there exist pairwise disjoint open sets $O_{n,i}\subset W$ with $L_{n,i}\Subset O_{n,i}\Subset W$.
Choose cutoff functions $\chi_{n,i}\in C^\infty_c(O_{n,i})$ with $\chi_{n,i}\equiv 1$ in a neighborhood of $L_{n,i}$ and such that $0\le \chi_{n,i}\le 1$.

Since each $L_{n,i}\subset K_0\subset F$ and $F$ is $C(e,\alpha)$-null,
thus for every $i$ there exists $\varphi_{n,i}\in C^\infty(\R^d)$ satisfying the conclusions of Lemma~\ref{lem:width} with
$$
E=L_{n,i},
\qquad
\|\varphi_{n,i}\|_{C^0(\mathbb{R}^d)}
\leq
\zeta
=\dfrac{\tau\,t^{n+4}}{1+\max_i \|D\chi_{n,i}\|_{C^0(W)}} .
$$
Choosing points $y_{n,i}\in L_{n,i}$, we can set $a_{n,i}\coloneqq r_n(y_{n,i})$ and define
$$
w_{n,i}\coloneqq a_{n,i}\,\chi_{n,i}\,\varphi_{n,i},
\qquad
u_{n+1}\coloneqq \sum_{i=1}^{N_n}w_{n,i}.
$$
Since the supports of $\chi_{n,i}$ are pairwise disjoint, the supports of $w_{n,i}$ are pairwise disjoint as well.
Therefore $u_{n+1}\in C^1_c(W)$ and
$$
\|Du_{n+1}\|_{C^0(W)} = \max_i \|Dw_{n,i}\|_{C^0(W)},
\qquad
\|u_{n+1}\|_{C^0(W)} = \max_i \|w_{n,i}\|_{C^0(W)}.
$$

By Leibniz' rule,
$$
Dw_{n,i}=a_{n,i}\bigl(\chi_{n,i}D\varphi_{n,i}+\varphi_{n,i}D\chi_{n,i}\bigr),
$$
hence
$$
\|Dw_{n,i}\|_{C^0(W)}
\le \|r_n\|_{C^0(K_n)}\bigl(c_\alpha+\|\varphi_{n,i}\|_{C^0(W)} \|D\chi_{n,i}\|_{C^0(W)}\bigr)
\le t^{n}M\bigl(c_\alpha+t^{n+4}\bigr),
$$
which is \eqref{eq:gradstep}.
Similarly,
$$
\|w_{n,i}\|_{C^0(W)}
\le |a_{n,i}|\,\|\varphi_{n,i}\|_{C^0(W)}
\le t^{n}M\cdot \tau t^{n+4}
= \tau\,t^{2n+4}M,
$$
so \eqref{eq:supstep} follows.

It remains to control the new residual on $K_{n+1}$.
Fix $x\in L_{n,i}$.
Since $\chi_{n,i}\equiv 1$ on a neighborhood of $L_{n,i}$, we have $D\chi_{n,i}(x)=0$.
Since the supports are disjoint, all other functions $w_{n,m}$ with $m\neq i$ vanish in a neighborhood of $x$.
Hence
$$
\partial_e u_{n+1}(x)=\partial_e w_{n,i}(x)=a_{n,i}\,\partial_e\varphi_{n,i}(x)=a_{n,i}.
$$
Therefore
$$
r_{n+1}(x)=r_n(x)-\partial_e u_{n+1}(x)=r_n(x)-a_{n,i}.
$$
Since $x,y_{n,i}\in E_{n,i}$, we obtain
$$
|r_{n+1}(x)|
\le \osc_{E_{n,i}}(r_n)
\le t^{n+1}M.
$$
Thus $\|r_{n+1}\|_{C^0(K_{n+1})}\le t^{n+1}M$, \eqref{eq:residual} is proved and the induction is complete.

To conclude the proof of the proposition, observe that, by \eqref{eq:supstep}, it holds
$$
\sum_{n=0}^\infty \|u_{n+1}\|_{C^0(U)}
\le \tau M \sum_{n=0}^\infty t^{2n+4}
\le \frac{\eta}{8}
< \eta.
$$
and, by \eqref{eq:gradstep},
$$
\sum_{n=0}^\infty \|Du_{n+1}\|_{C^0(U)}
\le M\sum_{n=0}^\infty t^{n}\bigl(c_\alpha+t^{n+4}\bigr)
\overset{\eqref{eq:t_geom_piccola}}{<} (1+\delta)c_\alpha M.
$$
Thus the series $\sum_{n=0}^\infty u_{n+1}$ converges, in the $C^1$-norm, to a function $u \in C^1_c(W)$ with
$$
\|u\|_{C^0(U)} < \eta,
\qquad
\|Du\|_{C^0(U)}\le (1+\delta)c_\alpha M.
$$
Finally set $K\coloneqq \bigcap_{n=0}^\infty K_n$.
Since $K_n$ are decreasing compact sets, $K$ is compact.
Moreover, by \eqref{eq:loss},
$$
\nu(U\setminus K)
\le \nu(U\setminus K_0)+\sum_{n=0}^\infty \nu(K_n\setminus K_{n+1})
< \frac{\eta}{2}+\sum_{n=0}^\infty \frac{\eta}{2^{n+2}}
=\eta.
$$
If $x\in K$, then $x\in K_n$ for every $n$, and by \eqref{eq:residual}
$$
|r_n(x)|=\left|\lambda(x)-\partial_e \sum_{k=1}^{n}u_k(x)\right|\le t^{n}M\to 0.
$$
Since $\sum_{k=1}^{n}u_k\to u$ in $C^1$ as $n \to +\infty$, we conclude that
$$
\partial_e u(x)=\lambda(x)
\qquad\forall x\in K,
$$
completing the proof.
\end{proof}

\section{Prescribed divergence}

We now derive the prescribed divergence theorem from the directional scalar lemma.

\begin{proof}[Proof of Theorem \ref{thm:div}]
Up to modifying $f$ on a set of arbitrarily small measure, we can assume that $f \in L^{\infty}(\Omega,\mu)$ and call $M \coloneqq \|f\|_{L^{\infty}(\Omega,\mu)}$.
% If $M=0$, then the conclusion is trivial with $V\equiv 0$, hence we assume $M>0$.

By Lusin's theorem there exists a compact set $C\subset \Omega$ such that $\mu(\Omega\setminus C)<\frac{\varepsilon}{3}$ and $f|_C$ is continuous.

Let us fix $\alpha \in \left(0,\frac{\pi}{2}\right)$ and $\Tilde{\delta}\in \left(0,\frac{1}{2}\right)$ such that
\begin{equation}\label{eq:alpha_delta_div}
    c_\alpha (1+\Tilde{\delta}) = \left(1+\frac{1}{\tan \alpha} \right)(1+\Tilde{\delta}) < 1+\delta
\end{equation}
and let $v_1,\dots,v_N$ be the directions given by Lemma~\ref{lem:finite-directions}.

For every $j=1,\dots,N$ define
$$
A_j\coloneqq \bigl\{x\in C:\ V(\mu,x)\cap C(v_j,\alpha)=\{0\}\bigr\}.
$$
Since $x\mapsto V(\mu,x)$ is Borel and the set
$$
\{L\in \Gr(k,d): L\cap C(v_j,\alpha)=\{0\}\}
$$
is open in the Grassmannian, each $A_j$ is Borel.
By Proposition~\ref{prop:nonfullbundle} and Lemma~\ref{lem:finite-directions}, it holds $\mu\Bigl(C\setminus \bigcup_{j=1}^N A_j\Bigr)=0$.
Make the family disjoint by setting
$$
B_1\coloneqq A_1,
\qquad
B_j\coloneqq A_j\setminus \bigcup_{m<j}A_m
\quad(j=2,\dots,N).
$$
Then the sets $B_j$ are Borel, pairwise disjoint, contained in $C$, and $\mu\Bigl(C\setminus \bigcup_{j=1}^N B_j\Bigr)=0$.
By inner regularity choose compact sets $C_j\subset B_j$ such that
$$
\mu\Bigl(C\setminus \bigcup_{j=1}^N C_j\Bigr)<\frac{\varepsilon}{3}.
$$
Since the $B_j$ are pairwise disjoint, the compact sets $C_j$ are pairwise disjoint as well.
Since the restriction of $f$ to the compact set $\bigcup_{j=1}^N C_j$ is continuous and $|f|\le M$ everywhere, by the Tietze extension theorem there exists a continuous function $\widetilde f \colon \Omega\to\R$ such that
$$
\widetilde f=f\quad\text{on }\bigcup_{j=1}^N C_j,
\qquad
\|\widetilde f\|_{C^0(\Omega)}\le M.
$$
We now fix pairwise disjoint open sets $U_j\Subset \Omega$ such that $C_j \subset U_j$, choose cutoff functions $\psi_j\in C^\infty_c(U_j)$ such that
$$
0\le \psi_j\le 1,
\qquad
\psi_j\equiv 1\quad\text{on }C_j,
$$
and set $\lambda_j\coloneqq \psi_j\widetilde f\in C_c(U_j)$. Then
$$
\lambda_j=f \quad\text{on }C_j,
\qquad
\|\lambda_j\|_{C^0(U_j)}\le M.
$$

Fix $j$.
By Lemma~\ref{lem:Cnull}, applied to $\mu$ on the set $C_j$, there exists a $C(v_j,\alpha)$-null set $F_j \subset C_j$ such that $\mu(C_j \setminus F_j)=0$. Considering the restricted measure $\nu_j\coloneqq \mu\llcorner C_j$, it holds $\nu_j(U_j\setminus F_j)=0$.

Now we apply Proposition~\ref{prop:local} with $\alpha,\Tilde{\delta}$ as fixed in \eqref{eq:alpha_delta_div} and
$$
U=U_j,\qquad \nu=\nu_j,\qquad e=v_j,
\qquad \lambda=\lambda_j,
\qquad
\eta_j\coloneqq \frac{\varepsilon}{3N},
$$
obtaining a compact set $K_j\subset U_j$ and a function $u_j\in C^1_c(U_j)$ such that
\begin{align*}
\nu_j(U_j\setminus K_j)&<\frac{\varepsilon}{3N},\\
\partial_{v_j}u_j(x)&=\lambda_j(x)\qquad \forall x\in K_j,\\
\|Du_j\|_{C^0(U_j)}&\le (1+\Tilde{\delta})c_\alpha M < (1+\delta)M,\\
\|u_j\|_{C^0(U_j)}&\le \varepsilon.
\end{align*}
Since $\nu_j$ is supported on $C_j$, replacing $K_j$ by $K_j\cap C_j$ if necessary, we may assume $K_j\subset C_j$.

To conclude, define
$$
V\coloneqq \sum_{j=1}^N u_j v_j.
$$
Since the supports of the $u_j$ are contained in the pairwise disjoint open sets $U_j$, the field $V$ belongs to $C^1_c(\Omega;\R^d)$.
Moreover
$$
DV(x)=Du_j(x)\otimes v_j
\qquad\text{if }x\in U_j,
$$
and $DV(x)=0$ outside $\bigcup_j U_j$.
Hence
$$
\|DV\|_{C^0(\Omega)}
=\max_j \|Du_j\|_{C^0(U_j)}
\leq (1+\delta) M.
$$
Therefore
$$
\Lip(V)\le (1+\delta) M.
$$
Similarly,
$$
\|V\|_{C^0(\Omega)}
=\max_j \|u_j\|_{C^0(U_j)}
\le \varepsilon.
$$

Set
$$
K\coloneqq \bigcup_{j=1}^N K_j.
$$
Since the $K_j$ are finitely many compact sets, $K$ is compact. Let us now fix $x \in K$; then there exists a unique $j \in \{1,\dots,N\}$ such that $x \in K_j$ and all other terms $u_i v_i$ in the definition of $V$, for $j \neq i$, vanish in a neighborhood of $x$.
Thus in a neighborhood of $x$, we have $V=u_j v_j$.
Writing $v_j=(\beta_1,\dots,\beta_d)$, it holds
$$
\diver V(x)=\sum_{k=1}^d \beta_k\partial_{e_k}u_j(x)= \partial_{v_j} u_j=\lambda_j(x)=f(x),
$$
% By Lemma~\ref{lem:div},
% $$
% \diver V(x)=\partial_{v_j}u_j(x)=\lambda_j(x)=f(x),
% $$
where the last equality holds
because $x\in K_j\subset C_j$ and $\lambda_j=f$ on $C_j$.
Therefore
$$
\diver V=f
\qquad\text{on }K.
$$

Finally,
\begin{align*}
\mu(\Omega\setminus K)
&\le \mu(\Omega\setminus C)
+\mu\Bigl(C\setminus \bigcup_{j=1}^N C_j\Bigr)
+\sum_{j=1}^N \mu(C_j\setminus K_j)\\
&< \frac{\varepsilon}{3}+\frac{\varepsilon}{3}+\sum_{j=1}^N \frac{\varepsilon}{3N}
=\varepsilon,
\end{align*}
completing the proof.
\end{proof}

\begin{rem}[Background fields]\label{rem:background-div}
If $W\in C^1(\Omega;\R^d)$ is any given vector field, applying Theorem~\ref{thm:div} to the scalar datum $f-\diver W$ yields a field $Z\in C^1_c(\Omega;\R^d)$ such that $\diver Z=f-\diver W$
outside an arbitrarily small exceptional set, with $\|Z\|_{C^{0}(\Omega)}$ arbitrarily small.
Then $V\coloneqq W+Z$ satisfies $\diver V=f$ outside an arbitrarily small exceptional set and is arbitrarily close to $W$ in the supremum norm.
\end{rem}

\section{Prescribed Jacobian}

We now prove the prescribed Jacobian theorem.

\begin{proof}[Proof of Theorem \ref{thm:jac}]
Up to modifying $g$ on a set of arbitrarily small measure, we can assume that $g \in L^{\infty}(\Omega,\mu)$ and let $L \coloneqq \|g-1\|_{L^{\infty}(\Omega,\mu)}$.

We choose again $\alpha,\Tilde{\delta}$ as in \eqref{eq:alpha_delta_div} and consider the finite family $v_1,\dots,v_N$ of directions given by Lemma~\ref{lem:finite-directions}.

As in the proof of Theorem~\ref{thm:div}, by Lusin's theorem there exists a compact set $C\subset \Omega$ such that $\mu(\Omega\setminus C)<\frac{\varepsilon}{3}$ and $g|_C$ is continuous. Similarly, we define the disjoint compact sets $C_j$ on which $V(\mu,\cdot) \cap C(v_j,\alpha)=\{0\}$ and the open, pairwise disjoint sets $U_j \Subset \Omega$ with $C_j \subset U_j$.

Applying the same arguments of the proof of Theorem \ref{thm:div} with the function $g-1$ in place of $f$, we obtain compact sets $K_j \subseteq C_j$ and functions $u_j \in C_c^1(U_j)$ such that
\begin{align*}
\mu(C_j\setminus K_j)&<\frac{\varepsilon}{3N},\\
\partial_{v_j}u_j(x)&=g(x)-1\qquad \forall x\in K_j,\\
\|Du_j\|_{C^0(U_j)}&\le (1+\Tilde{\delta})c_\alpha L < (1+\delta)L,\\
\|u_j\|_{C^0(U_j)}&\le \varepsilon
\end{align*}
and define
$$
\Phi(x)\coloneqq x+\sum_{j=1}^N u_j(x)v_j.
$$
Since the supports of the $u_j$ are contained in pairwise disjoint open sets, $\Phi$ belongs to $C^1(\Omega;\R^d)$ and $\Phi-\Id\in C^1_c(\Omega;\R^d)$.
Moreover
$$
\|\Phi-\Id\|_{C^0(\Omega)}
=\max_j \|u_j\|_{C^0(U_j)}
\le \varepsilon.
$$
Also,
$$
D\Phi(x)=I+v_j\otimes \nabla u_j(x)
\qquad\text{if }x\in U_j,
$$
and $D\Phi(x)=I$ outside $\bigcup_j U_j$.
Hence
$$
\|D(\Phi -\Id) \|_{C^0(\Omega)}
\le \max_j \|Du_j\|_{C^0(U_j)}
\le (1+\delta)L,
$$
so $\Lip(\Phi- \Id)\le (1+\delta)L$.

Set $K\coloneqq \bigcup_{j=1}^N K_j$. Let us fix $x \in K$ and, in order to show that $\det D\Phi=g$ on $K$, we observe that $x \in K_j$ for some $j$ and that $u_i(x)v_i=0$ for $i \neq j$ in a neighborhood of $x$. In this neighborhood it thus holds
$$
\Phi = \operatorname{Id} + u_jv_j,
\qquad
D\Phi=I_d +v_j \otimes \nabla u_j.
$$
Since the Jacobian is invariant under orthonormal changes of coordinates, we can choose one such that $v_j=(1,0,\dots,0)$ and simply compute
$$
\begin{aligned}
\det D\Phi(x)
=
1+ \partial_{v_j} u_j(x)
=
1+g(x)-1= g(x).
\end{aligned}
$$

Finally,
\begin{align*}
\mu(\Omega\setminus K)
&\le \mu(\Omega\setminus C)
+\mu\Bigl(C\setminus \bigcup_{j=1}^N C_j\Bigr)
+\sum_{j=1}^N \mu(C_j\setminus K_j)\\
&< \frac{\varepsilon}{3}+\frac{\varepsilon}{3}+\sum_{j=1}^N \frac{\varepsilon}{3N}
=\varepsilon.
\end{align*}
To prove the last assertion, by $(1+\delta)\|g-1\|_{L^\infty(\Omega,\mu)}< 1$, writing $\Phi=\Id + (\Phi-\Id)$ and using the triangular inequality, we have
$$
|\Phi(x)-\Phi(y)|\geq |x-y| - |x-y|\Lip(\Phi - \Id) 
\geq
|x-y|\left( 1- (1+\delta)\|g-1\|_{L^\infty(\Omega,\mu)}\right),
$$
which proves that $\Phi$ is injective and
$$
\Lip(\Phi^{-1})
\leq
\frac{1}{ 1- (1+\delta)\|g-1\|_{L^\infty(\Omega,\mu)}}.
$$
Since the extension of $\Phi$ to the whole $\mathbb{R}^d$ is a diffeomorphism with its image, its image is open in $\mathbb{R}^d$; moreover, this image is also closed in $\mathbb{R}^d$ because $\Phi$ coincides with the identity outside (the interior of) a compact set $E$ and $\Phi(E)$ is closed because compact. The connectedness of $\mathbb{R}^d$ yields the surjectivity of $\Phi$.
\end{proof}

We now show that the above result allows us to perturb any given diffeomorphism.

\begin{proof}[Proof of Corollary \ref{cor:diff_jac}]
Since $F$ is a diffeomorphism, the set $\Sigma$ is open; moreover, since $\mu\perp\LL^d$, also $\nu\coloneqq F_\#\mu$ satisfies $\nu\perp\LL^d$.

Define $h:\Sigma\to\R$ by
$$
h(y)\coloneqq (g\circ F^{-1})(y)\,\det DF^{-1}(y).
$$
and let us consider an open set $U \Subset \Sigma$ with $\nu(\Sigma\setminus U)< \frac{\varepsilon}{2}$.
Applying Theorem~\ref{thm:jac} on $U$ with measure $\nu$ and datum $h$, we find a compact set $K'\subset U$ and a map
$$
\Psi\in C^1(U;\R^d),\qquad \Psi-\Id\in C^1_c(U;\R^d),
$$
such that
$$
\nu(U\setminus K')<\min \left\{\frac{\varepsilon}{2}, \frac{\operatorname{dist}(U,\partial \Sigma)}{2} \right\},
\qquad
\det D\Psi(y)=h(y)\quad\forall y\in K',
$$
and
$$
\Lip(\Psi-\Id)\le (1+\delta)\|h-1\|_{L^\infty(\Sigma,\nu)},
\qquad
\|\Psi-\Id\|_{C^0(U)}\le \min \left\{\frac{\varepsilon}{2}, \frac{\operatorname{dist}(U,\partial \Sigma)}{2} \right\}.
$$
The last estimate in particular implies $\Phi(\Sigma)\subseteq \Sigma$.

Set
$$
\Phi\coloneqq \Psi\circ F,
\qquad
K\coloneqq F^{-1}(K').
$$
Then $\Phi\in C^1(\Omega;\Sigma)$. Since $F$ is a diffeomorphism and $K'$ is compact, the set $K$ is compact. Moreover,
$$
\mu(\Omega\setminus K)
=
\mu\bigl(F^{-1}(\Sigma\setminus K')\bigr)
=
\nu(\Sigma\setminus K')
<
\varepsilon.
$$

If $x\in K$, then $F(x)\in K'$, and by the chain rule
$$
\det D\Phi(x)
=
\det D\Psi(F(x))\,\det DF(x)
=
h(F(x))\,\det DF(x).
$$
By the definition of $h$,
$$
h(F(x))
=
g(x)\,\det DF^{-1}(F(x)),
$$
hence
$$
\det D\Phi(x)
=
g(x)\,\det DF^{-1}(F(x))\,\det DF(x)
=
g(x)
\qquad\forall x\in K.
$$

Finally,
$$
\Phi-F
=
(\Psi-\Id)\circ F,
$$
hence
$$
\supp(\Phi-F)\subset F^{-1}\bigl(\supp(\Psi-\Id)\bigr).
$$
Since $\supp(\Psi-\Id)\Subset\Sigma$ and $F^{-1}$ is continuous, it follows that $\Phi-F\in C^1_c(\Omega;\R^d)$. Moreover,
$$
\|\Phi-F\|_{C^0(\Omega)}
=
\|(\Psi-\Id)\circ F\|_{C^0(\Omega)}
=
\|\Psi-\Id\|_{C^0(\Sigma)}
\le \varepsilon.
$$
Also,
$$
\Lip(\Phi-F)
\le
\Lip(\Psi-\Id)\,\Lip(F)
\le
(1+\delta)\Lip(F)\|h-1\|_{L^\infty(\Sigma,\nu)}.
$$
Since $\nu=F_\#\mu$, we have
$$
\|h-1\|_{L^\infty(\Sigma,\nu)}
=
\|h\circ F-1\|_{L^\infty(\Omega,\mu)}.
$$
By the definition of $h$,
$$
h(F(x))
=
g(x)\,\det DF^{-1}(F(x))
=
\frac{g(x)}{\det DF(x)},
$$
therefore
$$
\|h-1\|_{L^\infty(\Sigma,\nu)}
=
\left\|\frac{g}{\det DF}-1\right\|_{L^\infty(\Omega,\mu)}.
$$
Thus
$$
\Lip(\Phi-F)\le (1+\delta)\Lip(F)
\left\|\frac{g}{\det DF}-1\right\|_{L^\infty(\Omega,\mu)},
$$
which is the desired estimate.
\end{proof}

\end{document}